\documentclass{article}

\usepackage[english]{babel}
\usepackage{amsthm,amsmath,amssymb,amstext}
\newtheorem{thm}{Theorem}
\newtheorem{lemma}[thm]{Lemma}
\theoremstyle{definition}
\newtheorem{definition}[thm]{Definition}
\begin{document}

\title{On the Jacobson element and generators of
the Lie algebra $\mathfrak{grt}$ in nonzero characteristic}
\author{Maria Podkopaeva}
\date{December 3, 2008}
\maketitle
\begin{abstract}
We state a conjecture (due to M. Duflo) analogous to the
Kashiwara--Vergne conjecture in the case of a characteristic $p>2$,
where the role of the Campbell--Hausdorff series is played by the
Jacobson element. We prove a simpler version of this conjecture
using Vergne's explicit rational solution of the Kashiwara--Vergne
problem. Our result is related to the structure of the
Grothendieck--Teichm\"{u}ller Lie algebra $\mathfrak{grt}$ in
characteristic $p$: we conjecture existence of a generator of
$\mathfrak{grt}$ in degree $p-1$, and we provide this generator for
$p=3$ and $p=5$.
\end{abstract}
\vskip 3em

Let $\mathfrak{lie}_2$ be a free Lie algebra over a field
$\mathbb{K}$ of characteristic zero with generators $x$ and $y$. It
is a graded Lie algebra
$\mathfrak{lie}_2=\prod_{k=1}^{\infty}\mathfrak{lie}_2^k$, where
$\mathfrak{lie}_2^k$ is spanned by the Lie words consisting of $k$
letters. We denote by $z=\mathrm{log}\mathrm e^x\mathrm e^y$ be the
Campbell--Hausdorff series:
\begin{equation}\label{ch}z=\sum_{n=1}^{\infty}z_n=
\sum_{n=1}^{\infty}\sum_{k=1}^{n}\frac{(-1)^{k-1}}{k}\sum_{(i,j)}
\frac{x^{i_1}y^{j_1}...x^{i_k}y^{j_k}} {i_1!j_1!...i_k!j_k!},
\end{equation}
where $i=(i_1,...,i_k)$, $j=(j_1,...,j_k)$, $i_m,j_m\in\mathbb Z$
and $i_m+j_m>0$ for all $m$, and $i_1+...+i_k+j_1+...+j_k=n$.

Let $\mathrm{Assoc}_2$ be the associative algebra with generators
$x$ and $y$, and let
$\tau:\mathfrak{lie}_2\rightarrow\mathrm{Assoc}_2$ be the natural
injection from the Lie algebra to its universal enveloping algebra.
Every element $a$ in $\mathrm{Assoc}_2$ admits a unique presentation
$a=a_0 + a_1 x + a_2 y$, where $a_0 \in \mathbb{Q}$ and $a_1,a_2 \in
\mathrm{Assoc}_2$. We shall denote $a_1 = \partial_x a, a_2=
\partial_y a$.

We define the graded vector space of circular words
$\mathfrak {tr}_2$ as the quotient
$$\mathfrak {tr}_2=\mathrm{Assoc}_2^+/\langle(ab-ba),\ a,b\in \mathrm{Assoc}_2\rangle,$$
where $\mathrm{Assoc}_2^+=\prod_{k=1}^{\infty}\mathrm{Assoc}^k(x,y)$
and $\langle(ab-ba),\ a,b\in \mathrm{Assoc}_2\rangle$ is the
subspace of $\mathrm{Assoc}_2^+$ spanned by commutators. We denote
by $\tilde{\mathrm{tr}}:\mathrm{Assoc}_2\rightarrow\mathfrak{tr}_2$
the corresponding natural projection. Let $\mathfrak{g}$ be a Lie
algebra over $\mathbb{K}$, and let $\rho: \mathfrak{g} \rightarrow
{\rm End}(V)$ be a finite dimensional representation. Then, each
element $\tilde{\mathrm{tr}}(a) \in \mathfrak{tr}_2$ gives rise to a
map $\rho_a: \mathfrak{g}\times \mathfrak{g} \rightarrow \mathbb{K}$
defined by the formula $\rho_a(x,y)={\rm Tr}(\rho(a(x,y))$.

The Kashiwara--Vergne conjecture [8] (now a theorem [1]) is an
important problem of Lie theory which in particular implies the
Duflo isomorphism [4] between the center of the universal enveloping
algebra  and the ring of invariant polynomials. The conjecture
states that there exist elements $F(x,y)$ and $G(x,y)$ in
$\mathfrak{lie}_2$ such that
\begin{equation}\label{KV1}
x+y-\mathrm {log}\mathrm e^x\mathrm e^y=(\mathrm e^{\mathrm{ad}
x}-1)F(x,y)+(1-\mathrm e^{-\mathrm{ad} y})G(x,y)
\end{equation}
and
\begin{equation}\label{KV2}
\tilde{\mathrm{tr}}(x\partial_xF+
y\partial_yG)=\frac{1}{2}\tilde{\mathrm{tr}}\left(\frac{x}{\mathrm
e^{x}-1}+ \frac{y}{\mathrm e^{y}-1}-\frac{z}{\mathrm
e^{z}-1}-1\right).
\end{equation}

Since the statement of the Kashiwara--Vergne conjecture uses the
exponential function, it can only be defined over a field of
characteristic zero. Michel Duflo [5] suggested the following
question which resembles the Kashiwara--Vergne conjecture in the
case of a positive characteristic. Let $p>2$ be a prime, and let
$\mathbb{K}$ be a field of characteristic $p$.

\vskip 1em \noindent\textbf{Conjecture 1} \, \textit{There exist
$A(x,y)$ and $B(x,y)$ in $\mathfrak {lie}_2$ over $\mathbb{K}$  such
that
\begin{equation}\label{partial KV1}
[x,A(x,y)]+[y,B(x,y)]=x^p+y^p-(x+y)^p
\end{equation}
and
\begin{equation}\label{partial KV2}
\tilde{\mathrm{tr}}(x\partial_x A+y\partial_y
B)=\frac{1}{2}\tilde{\mathrm{tr}}(x^{p-1}+y^{p-1}-(x+y)^{p-1}).
\end{equation}
}

Note that $x^p+y^p-(x+y)^p$ is the Jacobson element (see {\em e.g.}
[9]) in $\mathfrak{lie}_2$ over $\mathbb{K}$. We will prove a
simplified version of Conjecture 1. For an arbitrary element
$a=x_{i_1}...x_{i_n}\in\mathrm{Assoc}_2$, we put
$a^T=(-x_{i_n})...(-x_{i_1})$. Consider the quotient of
$\mathfrak{tr}_2$ by the relations
$\tilde{\mathrm{tr}}(a)=\tilde{\mathrm{tr}}(a^T)$. We denote by
$\mathrm{tr}$ the projection from $\mathrm{Assoc}_2$ to the above
quotient. Let $\mathfrak{g}$ be a Lie algebra over $\mathbb{K}$ and
let $\rho$ be a finite dimensional representation of $\mathfrak{g}$
with the property $\rho(x)^t=-\rho(x)$ (here $\rho(x)^t$ stands for
a transposed matrix). Then, the map $\rho_a$ only depends on
$\mathrm{tr}(a)$. For instance, that is the case of the adjoint
representation of the quadratic Lie algebra (a Lie algebra equipped
with a non-degenerate invariant symmetric bilinear form). Hence, we
refer to the "quadratic" case of  Conjecture 1.

\begin{proof} \emph{(in the quadratic case)}\ \ We use the following simple
facts from number theory.
\begin{lemma}\emph{(Wilson's Theorem)}\label{Wilson}
$(p-1)!=-1\mod p.$
\end{lemma}
Let $B_m$ be the $m$-th Bernoulli number.
\begin{lemma}\label{Staudt}Let $p$ be a prime and $m$
be an even number. If $(p-1)\nmid m$, then $B_m$ is a $p$-integer.
If $(p-1)|m$, then $pB_m$ is a $p$-integer and $pB_m=-1\mod
p$.
\end{lemma} See [7] for the proofs of these lemmas.

In [10] Vergne gave the following explicit solution of the
Kashiwara--Vergne problem (in the quadratic case). Consider the
functions
$$\Theta(t)=\frac{1-\mathrm e^{-t}}{t},\ \ \ \ R(t)=\frac{\mathrm e^t-\mathrm e^{-t}-2t}{t^2}.$$
Let $\mathcal R$ be the derivation of $\mathfrak {lie}_2$ such that
$\mathcal R|_{\mathfrak
{lie}_2^n}=n\mathrm{Id}|_{\mathfrak{lie}_2^n}$. The solutions of the
Kashiwara--Vergne problem are given by
$$F(x,y)=-\Theta(-\mathrm{ad} x)^{-1}U(x,y),\ \ G(x,y)=-\Theta(-\mathrm{ad} y)^{-1}V(x,y),$$
where $U$ and $V$ are defined by the equations
\begin{equation}(\mathcal
R+1)U(x,y)=-\frac{1}{2}\Theta(\mathrm{ad} x)\Theta(\mathrm{ad}
z)^{-1}R(\mathrm{ad} z) (\Theta(-\mathrm{ad}
z)^{-1}x+\Theta(\mathrm{ad} z)^{-1}y)+\frac{1}{2}\Theta(-\mathrm{ad}
x)y,\end{equation}
\begin{equation}(\mathcal
R+1)V(x,y)=-\frac{1}{2}\Theta(-\mathrm{ad} y)\Theta(-\mathrm{ad}
z)^{-1}R(\mathrm{ad} z) (\Theta(-\mathrm{ad}
z)^{-1}x+\Theta(\mathrm{ad} z)^{-1}y)-\frac{1}{2}\Theta(\mathrm{ad}
y)x.\end{equation} Let $p$ be a prime. It is easy to show that the
lowest homogeneous degree in $x$ and $y$ of a term of $F$ with
non-$p$-integer coefficient is $p-1$. We expand the $(p-1)$-st
homogeneous component $F_{p-1}$ of $F$ in powers of $p$:
$F_{p-1}=\sum_{n=-\infty}^{\infty}f_np^n$. It is easy to see that
the lowest power of $p$ in this expansion is $-2$: the $\frac{1}{p}$
coming from $R(\mathrm{ad} z)=2(\frac{\mathrm{ad}
z}{3!}+\frac{(\mathrm{ad} z)^3}{5!}+\frac{(\mathrm{ad}
z)^5}{7!}+...)$ is multiplied by the $\frac{1}{p}$ coming from the
inverse of $\mathcal R+1$ . However, the following computation shows
that the coefficient $f_{-2}$ is actually equal to zero. By the
definition of the Campbell--Hausdorff series ($\ref{ch}$) we see
that $z_1=x+y$. We have
$$\frac{f_{-2}}{p^2}=-\frac{1}{2}\cdot\frac{1}{p}\cdot\frac{(\mathrm{ad}
z_1)^{p-2}}{p!}(x+y),$$ and so
$$f_{-2}=-\frac{1}{2}\cdot\frac{(\mathrm{ad}
z_1)^{p-2}}{(p-1)!}(x+y)=-\frac{1}{2}\cdot\frac{(\mathrm{ad}
(x+y))^{p-2}}{(p-1)!}(x+y)=0.$$ Thus, the $p$-adic expansion of
$F_{p-1}$ has the form
$F_{p-1}=\frac{f_{-1}}{p}+\sum_{n=0}^{\infty}f_np^n$. The same
calculation for $G$ gives
$G_{p-1}=\frac{g_{-1}}{p}+\sum_{n=0}^{\infty}g_np^n$. Consider the
$p$-th homogeneous part of equation $(\ref{KV1})$. The left-hand
side yields
$$\frac{\mathrm{ad} x}{1!}F_{p-1}+\frac{(\mathrm{ad} x)^2}{2!}F_{p-2}+...+
\frac{(\mathrm{ad} x)^{p-1}}{(p-1)!}F_1+
\frac{-\mathrm{ad} y}{1!}G_{p-1}+\frac{(-\mathrm{ad}
y)^2}{2!}G_{p-2}+...+\frac{(-\mathrm{ad} y)^{p-1}}{(p-1)!}G_1,$$ and the
right-hand side is of the form
$$-\sum_{k=1}^{p}\frac{(-1)^{k-1}}{k}\sum_{(i,j)}\frac{x^{i_1}y^{j_1}...x^{i_k}y^{j_k}}
{i_1!j_1!...i_k!j_k!},$$ where $i=(i_1,...,i_k)$, $j=(j_1,...,j_k)$,
$i_m,j_m\in\mathbb Z$ and $i_m+j_m>0$ for all $m$, and
$i_1+...+i_k+j_1+...+j_k=p$. Expanding the above expressions in
powers of $p$ and comparing the coefficients at $\frac{1}{p}$, we
have
$$\mathrm{ad} x\cdot f_{-1}+\mathrm{ad} y\cdot g_{-1}=-\frac{x^p+y^p}{(p-1)!}-(x+y)^p\mod p.$$
Using Lemma $(\ref{Wilson})$, we obtain
$$\mathrm{ad} x\cdot f_{-1}+\mathrm{ad} y\cdot g_{-1}=x^p+y^p-(x+y)^p\mod p.$$

Next, consider equation
$(\ref{KV2})$. It is easy to see that the lowest
homogeneous degree in $x$ and $y$ with non-$p$-integer coefficient
is $p-1$. Consider the $(p-1)$-st homogeneous part of the equation
and expand it in powers of $p$. Comparing the coefficients at
$\frac{1}{p}$, we obtain
$$\mathrm{tr}(x\partial_x f_{-1}+y\partial_y
g_{-1})=\frac{pB_{p-1}}{2(p-1)!}(x^{p-1}+y^{p-1}-(x+y)^{p-1})\mod
p.$$ Using Lemma $(\ref{Staudt})$ we obtain
$$\mathrm{tr}(x\partial_x f_{-1}+ y\partial_y
g_{-1})=\frac{1}{2}(x^{p-1}+y^{p-1}-(x+y)^{p-1})\mod p.$$ We put
$A(x,y)=f_{-1}(x,y)$ and $B(x,y)=g_{-1}(x,y)$.
\end{proof}

{\em Remark.} In order to use a similar strategy for proving
Conjecture 1 in the general case we need a control of the $1/p$
behavior of coefficients of a rational solution of the
Kashiwara--Vergne conjecture. The solution of [1] uses Kontsevich
integrals over configuration spaces, and {\em a priori} it is
defined over $\mathbb{R}$. The existence of rational solution
follows by linearity, but there is no control over coefficients.

\vskip 1.5em In [2], the Kashiwara--Vergne problem was related to
the theory of Drinfeld's associators. By analogy, this relation
suggests a link between Conjecture 1 and the structure of the
Grothendieck--Teichm\"uller Lie algebra over a field of
characteristic $p$.

\begin{definition}The algebra $\mathfrak t_n$ is the quotient of the free Lie algebra with
$n(n-1)/2$ generators $t^{i,j}=t^{j,i}$ by the following relations
\begin{equation}\label{t4-1}[t^{i,j},t^{k,l}]=0
\end{equation}
if all indices $i,j,k$, and $l$ are distinct, and
\begin{equation}\label{t4-2}[t^{i,j}+t^{i,k},t^{j,k}]=0
\end{equation}
for all triples of distinct indices $i,j$, and $k$.
\end{definition}
Below we will use the following statement (see [3]).
\begin{lemma}\label{t4}$\mathfrak t_4\cong\mathbb Kt^{1,2}\oplus
\mathfrak{lie}(t^{1,3},t^{2,3})\oplus\mathfrak{lie}(t^{1,4},t^{2,4},t^{3,4})$,
where $\mathfrak{lie}(t^{1,4},t^{2,4},t^{3,4})$ is an ideal acted by
$\mathbb Kt^{1,2}\oplus \mathfrak{lie}(t^{1,3},t^{2,3})$, and
$\mathfrak{lie}(t^{1,3},t^{2,3})$ is an ideal in
$\mathbb Kt^{1,2}\oplus \mathfrak{lie}(t^{1,3},t^{2,3})$ acted
by $\mathbb Kt^{1,2}$.

\end{lemma}
\begin{definition} The Grothendieck--Teichm\"{u}ller Lie algebra is the Lie
algebra spanned by the elements $\psi\in\mathfrak{lie}_2$ satisfying
the following relations:
\begin{equation}\label{grt1}
\psi(x,y)=-\psi(y,x),
\end{equation}
\begin{equation}\label{grt2}
\psi(x,y)+\psi(y,z)+\psi(z,x)=0,
\end{equation}
where $z=-x-y$,
\begin{equation}\label{grt3}
\psi(t^{1,2},t^{2,34})+\psi(t^{12,3},t^{3,4})=\psi(t^{2,3},t^{3,4})+
\psi(t^{1,23},t^{23,4})+\psi(t^{1,2},t^{2,3}),
\end{equation}
where the latter takes place in $\mathfrak t_4$ and
$t^{i,jk}=t^{i,j}+t^{i,k}$.
\end{definition}
The Deligne--Drinfeld conjecture [3] states that, over a field of
characteristic zero, $\mathfrak{grt}$ is a graded free Lie algebra
with generators $\sigma_{2n-1}, n=1,2,\dots$ of degree ${\rm
deg}(\sigma_{2n-1})=2n-1$ . This conjecture is numerically verified
up to degree 16. Consider the algebra $\mathfrak{grt}$ over a field
of characteristic $p>2$. Conjecture 2 (see below) suggests existence
of a generator of $\mathfrak{grt}$ in the degree $p-1$.

Consider the function $\psi(x,y)$ such that $\psi(-x-y,x)=A(x,y)$
and $\psi(-x-y,y)=B(x,y)$, where $A$ and $B$ are solutions of
(\ref{partial KV1},\ref{partial KV2}). Such a function exists
because the solutions of (\ref{partial KV1},\ref{partial KV2}) can
always be chosen symmetric: $A(x,y)=B(y,x)$.

We define another grading on $\mathfrak{lie}_2$. The depth of a Lie
monomial is defined as the number of $y$'s entering this monomial.
The depth of a Lie polynomial is the smallest depth of its
monomials.

\begin{lemma}\label{depth} The polynomial $\psi(x,y)$ is of depth one.
\end{lemma}
\begin{proof} By definition, we have $\psi(x,y)=A(y,-x-y)$, so we
must prove that $A(y,-x-y)$ is of depth one, i.e., that $A(x,y)=c \,
\mathrm{ad}_y^{p-2}x+...$, where $c\neq0$. In equation (\ref{partial
KV1}), we consider the homogeneous part of degree $p-2$ in $y$:
$$[x,A_{xy^{p-2}}]+[y,B_{x^2y^{p-1}}]=(x^p+y^p-(x+y)^p)_{x^2y^{p-2}},$$
where the indices $x^iy^{p-i}$ denote the corresponding homogeneous
degree parts of the expressions. By the definition (\ref{ch}) of the
Campbell--Hausdorff series, we have
$A_{xy^{p-2}}=c\cdot\mathrm{ad}_y^{p-2}x$ and $B_{x^2y^{p-1}}=\sum
b_{ql}\mathrm{ad}_y^{l_1}\mathrm{ad}_x\mathrm{ad}_y^{l_2}\mathrm{ad}_xy$,
where the sum is taken over all $l_1$ and $l_2$ such that $l_1\geq0$, $l_2\geq0$, and $l_1+l_2=p-2$. \\
Suppose $c=0$. Then
$(x^p+y^p-(x+y)^p)_{x^2y^{p-2}}=[y,B_{x^2y^{p-1}}]$. Consider the
injection $\tau$ to the universal enveloping algebra. The image
under $\tau$ of the right-hand side of the above equation is a sum
of monomials either beginning or ending with $y$, so it does not
contain the monomial $xy^{p-2}x$, whereas the left-hand side of this
equation does contain such a monomial. Thus, $c\neq0$, and so $\psi$
is of depth one.
\end{proof}

 \noindent\textbf{Conjecture 2}\ \ The function $\psi(x,y)$
belongs to $\mathfrak{grt}$. \vskip 1em
We verify this conjecture for $p=3$ and $p=5$. \\ \\
\textit{The case $p=3$}. \\ \\
The solution of (\ref{partial KV1},\ref{partial KV2}) is given by
$A(x,y)=-[x,y]$ and $B(x,y)=[x,y]$. Then $\psi(x,y)=-[x,y]$. We
verify conditions (\ref{grt1}-\ref{grt3}).\\ \\
Condition (\ref{grt1}). We have
$-\psi(y,x)=[y,x]=-[x,y]=\psi(x,y)$.\\ \\
Condition (\ref{grt2}). We have
$\psi(x,y)+\psi(y,-x-y)+\psi(-x-y,x)=
-[x,y]-[y,-x-y]-[-x-y,x]=-[x,y]+[y,x]+[y,x]=-3[x,y]=0\mod 3.$ \\ \\
Condition (\ref{grt3}). We write $(ij)$ for $t^{i,j}$. We have
$\psi((12),(23)+(24))+\psi((13)+(23),(34))-\psi((23),(34))-\psi((12)+(13),(24)+(34))
-\psi((12),(23))=-[(12),(23)+(24)]-[(13)+(23),(34)]+[(23),(34)]+[(12)+(13),(24)+(34)]
+[(12),(23)]=0.$\\ \\
\textit{The case $p=5$}. \\ \\
The solution of (\ref{partial KV1},\ref{partial KV2}) is given by
$A(x,y)=[x,[x,[x,y]]]+[y,[x,[x,y]]]+2[y,[y,[x,y]]]$ and
$B(x,y)=[y,[y,[y,x]]]+[x,[y,[y,x]]]+2[x,[x,[y,x]]]$. Then
$\psi(x,y)=2[x,[x,[x,y]]]+2[y,[y,[x,y]]]+3[y,[x,[x,y]]]$. Let us
verify the conditions (\ref{grt1}-\ref{grt3}).\\ \\
Condition (\ref{grt1}). We have
$-\psi(y,x)=-2[y,[y,[y,x]]]-2[x,[x,[y,x]]]-3[x,[y,[y,x]]]=
2[y,[y,[x,y]]]+2[x,[x,[x,y]]]+3[x,[y,[x,y]]]=\psi(x,y)$.\\ \\
By direct calculation, we obtain that $\psi[x, y] + \psi[y, -x - y]
+
\psi[-x - y, x]=0 \mod 5$, which gives (\ref{grt2}).\\ \\
A lengthy calculation using Lemma (\ref{t4}) gives equation
(\ref{grt3}) modulo 5. \hfill$\Box$ \vskip 1 em The kernel of the
projection
$\pi:\mathfrak{grt}\rightarrow\mathfrak{grt}/[\mathfrak{grt},\mathfrak{grt}]$
contains only the elements of depth greater or equal to 2 (see [6]
for details). Thus, Conjecture 2 together with Lemma \ref{depth}
would give a generator of $\mathfrak{grt}$ in degree $p-1$.

\vskip 2em\textbf{Acknowledgements:} I would like to thank Anton
Alekseev for suggesting the problem and for numerous discussions.

\section*{References}

  \noindent[1] \ Alekseev A., Meinrenken E.,
  \emph{On the Kashiwara--Vergne conjecture}.
  Invent. math. 164, 615–634, (2006).
\vskip 1em
  \noindent[2] \ Alekseev A., Torossian Ch.,
  \emph{The Kashiwara--Vergne conjecture and Drinfeld's
  associators}.
  arXiv:0802.4300
\vskip 1em
  \noindent[3] \ Drinfel'd V.G.,
  \emph{On quasiriangular quasi-Hopf algebras and a group closely connected
  with $\mathrm{Gal}(\bar{\mathbb Q},\mathbb Q)$}.
  Leninrgad Math. J. Vol. 2 (1991), No. 4.
\vskip 1em
  \noindent[4] \ Duflo M.,
  \emph{Op\'erateurs diff\'erentiels bi-invariants sur un groupe de
  Lie}.
  Ann. Sci. \'Ecole Norm. Sup. 10 (1977), 265-288.
\vskip 1em
  \noindent[5] \ Duflo M., private communications.
\vskip 1em
  \noindent[6] \ Ihara Y.,
  \emph{Some arithmetic aspects of Galois actions in the pro-$p$ fundamental group of
  $\mathbb{P}^1-\{0,1,\infty\}$}.
  Arithmetic fundamental groups and noncommutative algebra (Berkley, CA,
  1999), 247-273, Proc. Sympos. Pure Math., 70, Amer. Math. Soc.,
  Providence, RI, 2002.
\vskip 1em
  \noindent[7] \ Ireland K., Rosen M.
  \emph{A classical introduction to modern number theory}.
  New York, Heidelberg, Berlin: Springer-Verlag, 1982.
\vskip 1em
  \noindent[8] \ Kashiwara M., Vergne M.,
  \emph{The Campbell-Haussdorff formula and invariant hyperfunctions}.
  Invent. Math. 47 (1978) 249-272.
\vskip 1em
  \noindent[9] \ Mathieu O.
  \emph{Classification des alg\`{e}bres de Lie simples}.
  S\'{e}minaire N. Bourbaki, 1998-1999, exp. No. 858.
\vskip 1em
  \noindent[10] \ Vergne M.,
  \emph{Le centre de l'alg\`{e}bre enveloppante et la formule de Campbell--Hausdorff}.
  C. R. Acad. Sci. Paris S\'{e}r. I Math. 329 (1999), no. 9, 767-772.

\begin{center}
\vskip 2em \textsc{Maria Podkopaeva}

Universit\'e de Gen\`eve

Section de math\'ematiques

2-4 rue du Li\`evre, CP 64

1211 Gen\`eve 4, Switzerland

Maria.Podkopaeva@unige.ch \end{center}

\end{document}